\documentclass[10pt]{article}
\font\smallit=cmti10

\usepackage[english]{babel}
\usepackage{amssymb,amsthm,amsmath,amsfonts}
\usepackage{pst-all,pst-3dplot,pstricks,pstricks-add,pst-math,pst-xkey}
\usepackage{graphicx}
\usepackage{latexsym}
\usepackage{ifthen, bigstrut}
\usepackage{accents}
\usepackage[english]{babel}
\usepackage{hyperref}
\usepackage{breakurl}
\usepackage[T1]{fontenc}
\usepackage{ifthen, bigstrut}

\usepackage{bbold}
\usepackage{graphs}
\usepackage{enumerate}
\usepackage{epsfig}
\usepackage{subfigure}
\usepackage{skak}
\usepackage{chessboard}
\usepackage{chessfss}
\usepackage{lscape}

\newtheorem{theorem}{Theorem}

\newtheorem{corollary}[theorem]{Corollary}
\newtheorem{lemma}[theorem]{Lemma}

\theoremstyle{definition}
\newtheorem{definition}{Definition}

\newtheorem{example}{Example}

\newcommand{\GL}{G^\mathcal{L}}
\newcommand{\GR}{G^\mathcal{R}}
\newcommand{\HL}{H^\mathcal{L}}
\newcommand{\HR}{H^\mathcal{R}}
\newcommand{\XL}{X^\mathcal{L}}
\newcommand{\XR}{X^\mathcal{R}}

\def\cgstar{\mathord{\ast}}

\newcommand{\LS}{\mathit{Ls}}

\newcommand{\LSu}{\underline{\mathit{Ls}}}
\newcommand{\RS}{\mathit{Rs}}
\newcommand{\RSo}{\overline{\mathit{Rs}}}

\newcommand{\GB}{\mathbb{GS}}

\newcommand{\Num}[1]{\widehat{#1}}

\begin{document}

\begin{center}
\uppercase{\bf When waiting moves you\\ in scoring combinatorial games}
\vskip 20pt
{\bf Urban Larsson\footnote{Supported by the Killam Trust}}\\
{\smallit Dalhousie University, Canada}\\
{\bf Richard J.~Nowakowski}\\
{\smallit Dalhousie University, Canada}\\
{\bf Carlos P. Santos\footnote{Corresponding author: Centro de Estruturas Lineares e Combinatórias, Faculdade de Ciências, Universidade de Lisboa, Campo Grande, Edifício C6, Piso 2, 1749-016 Lisboa, Portugal; cmfsantos@fc.ul.pt}}\\
{\smallit Center for Linear Structures and Combinatorics, Portugal}\\
\end{center}

\begin{abstract}
Combinatorial Scoring games, with the property `extra pass moves for a player does no harm',
are characterized. The characterization involves an order embedding of Conway's Normal-play games.
Also, we give a theorem for comparing games with scores (numbers) which extends Ettinger's work on dicot Scoring games.
\end{abstract}

\section{Introduction}
\textit{The Lawyer's offer:}
To settle a dispute, a court has ordered you and your opponent to play a
Combinatorial game, the winner (most number of points) takes all.
Minutes before the contest is to begin, your opponent's lawyer approaches you
 with an offer: "You, and you alone, will be allowed a pass move to use once,
at any time in the game,
 but you must use it at some point (unless the other player runs out of moves before you used it)."
Should you accept this generous offer?

We will show when you should accept and when you should decline the offer.
It all depends on whether Conway's Normal-play games (last move wins) can be embedded
in the `game' in an order preserving way.

\medskip
\textit{Combinatorial games} have perfect information, are played by two players who move alternately,
but moreover, the games
finish regardless of the order of moves. When one of the players cannot move,
the winner of the game is declared by some predetermined winning condition.
The two players are usually called Left (female pronoun) and Right (male pronoun).

Many combinatorial games have the property that
the game decomposes into independent sub-positions. A player then has the choice of playing
in exactly one of the sub-positions; the whole position is the \textit{disjunctive sum}
of the sub-positions. The disjunctive sum of positions $G$ and $H$ is written $G+H$.
Such \emph{additive} \cite{Johns2014} games include \textsc{go, domineering, konane, amazons, dots\&boxes} and
also (end-positions of) \textsc{chess} but do not include \textsc{hex} or any type of Maker-Maker and Maker-Breaker games.
See \cite{Beck2006}, for example, for techniques to analyse these latter games.

\textit{Normal-play} games have the last player
to move as the winner; \textit{Mis\`ere-play} games
have that player as the loser. In this paper the focus is on
\textit{Scoring games} in which the player with the greatest score wins.

Finding general results for Scoring games has proven difficult. There are only five contributors known to the authors:
Milnor \cite{Milno1953}, followed by Hanner \cite{Hanne1959},
considered games with no Zugzwang positions; Johnson \cite{Johns2014} abstracted
 from a game played with knots; Ettinger \cite{Ettin1996,Ettin2000} considered
\textit{dicots}, that is games where either both players have a move or neither does;
Stewart \cite{Stewa2011} considered a very general class of games.
We give an overview of
their results in Section  \ref{sec:survey}.

Aviezri Fraenkel coined the terms `Math' games and `Play' games.
The former have properties that mathematicians like. On the other hand,
Play games tend to be harder to analyze, for example \textsc{go, dots-\&-boxes, othello, blokus}
and \textsc{kulami}; moreover they give direction to the mathematical research. `Play' Scoring games
tend to have some common strategic considerations. This paper focuses on three.

\begin{itemize}
\item \textit{Zugzwang} (German for ``compulsion to move'') is a situation where one player is
   put at a disadvantage because he has to make a move when he would prefer to pass
   and make no move.
\item \textit{Bonus/penalty}: In many Scoring games, there are penalties or
bonuses to be awarded when a game finishes.
\item \textit{Greediness principle}: Given two games $G$ and $H$,
Left prefers a game $G$ for a game $H$,
whenever each Left option of $H$ is also a Left option in $G$,
and each Right option of $G$ is also a Right option of  $H$.
\end{itemize}

Zugzwang and the Greediness principle relate to the question posed in the Lawyer's offer. Perhaps one would
believe that if all other things remain equal,  giving Left an extra option is an advantage,
at least no disadvantage, to Left. Surprisingly, this is not always true, indeed it is not true
in \cite{Stewa2011}, nor is it true in Mis\`ere-play games. If there are Zugzwangs in the `game', then you would be inclined to accept, but, as we will see, this does not reveal the whole truth. See Figure \ref{fig:sample} for an example.

Classes of Scoring games $\mathcal{S}$,
like Normal and Mis\`ere-play games, have a defined equivalence,  $\equiv$,
(often called `equality') which gives
rise to equivalence classes, and where $S/\!\!\equiv$ forms a monoid. In Normal-play games
this gives an ordered abelian group.
For the class of all Mis\`ere-play games the monoid has little structure.
Significant results concerning  Mis\`ere-play games only became possible
after Plambeck and Plambeck \& Siegel (see \cite{PlambS2008})
pioneered the approach of restricting the total set of games under consideration.

As shown in Stewart \cite{Stewa2011}, the monoid based on the full class of Scoring games
also has little structure. Following Plambeck \& Siegel's approach, we restrict the subset
of Scoring games under consideration to obtain a monoid with a useful structure.

In Section \ref{sec:terminology} we formally develop the concepts needed for Scoring games,
together with some basic results. In Section \ref{sec:Normal-play} we give some Normal-play background; see \cite{AlberNW2007, Siege2013}
for more on Normal-play games.

Our main results are given in Section \ref{sec:obstacle}.
In Theorem \ref{thm:emb}, we study when Normal-play games can be embedded in a family of Scoring games in an order preserving way.
For games $G$ and $H$, to check if $G\geqslant H$ (see Definition \ref{def:equality}) involves comparisons using all scoring games. Theorem \ref{thm:comp} gives a method that avoids this if one of the games is a number, and in particular answers the question: `who wins?', i.e. is $G\geqslant 0?$

To illustrate the concepts, we refer to games based on \textsc{konane} (see also Section~\ref{sec:ScoringGames}).

\subsection{\textsc{konane}}
 \textsc{konane} is a traditional Hawaiian Normal-play game, played on a
$m\times n$ checker-board with white stones on the white squares and
black on the black squares with some stones removed.
Stones move along rows or columns but not both in the same move.
A stone can jump over an adjacent opponent's stone provided that there is
an empty square on the other side, and the opponent's stone is removed.
Multiple jumps are allowed on a single move but are not mandatory.
When the player to move has no more options then the game is over and, in Normal-play,
the player is the loser.

\textsc{scoring-konane} is played as \textsc{konane}, but when the player to move
has no options then the game is over, and the score is
\textit{the number of stones Left has removed minus the number of stones Right has removed.} Left wins if the score is positive, loses if it is negative and ties if it is zero.

\begin{figure}[ht!]
\begin{center}
\scalebox{.65}{
\psset{unit=1 cm}
\begin{pspicture}(2,0)(3,4)
\psline(0,0)(0,3)(5,3)(5,0)(0,0)
\psline(1,0)(1,3)\psline(2,0)(2,3)\psline(3,0)(3,3)\psline(4,0)(4,3)
\psline(0,1)(5,1)\psline(0,2)(5,2)
\rput{0}(1.50,2.50){\scalebox{2.9}{$\bullet$}}\rput{0}(1.50,1.50){\scalebox{2.9}{$\circ$}}\rput{0}(0.50,0.50){\scalebox{2.9}{$\circ$}}\rput{0}(3.50,0.50){\scalebox{2.9}{$\bullet$}}
\end{pspicture}
\hspace{5 cm}
\begin{pspicture}(4,0)(1,4)

\psline(0,0)(0,1)(5,1)(5,0)(0,0)
\psline(1,0)(1,1)\psline(2,0)(2,1)\psline(3,0)(3,1)\psline(4,0)(4,1)

\rput{0}(0.50,0.50){\scalebox{2.9}{$\circ$}}\rput{0}(1.50,0.50){\scalebox{2.9}{$\bullet$}}\rput{0}(3.50,0.50){\scalebox{2.9}{$\bullet$}}

\end{pspicture}
}
\caption{To the left, a {\sc scoring-konane} Zugzwang, for which the Lawyer's offer should be accepted. In the right-most game, Black should reject it playing first.}\label{fig:sample}
\end{center}
\end{figure}
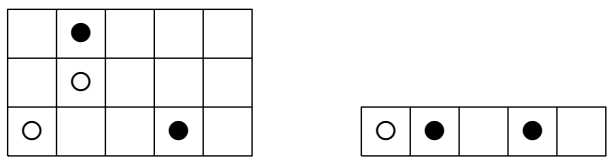

In Figure \ref{fig:sample}, we show that it is not clear whether you should accept the Lawyer's offer if you only get beforehand information that the `game' to play is {\sc scoring-konane}, but not which particular position. In Section~\ref{sec:ScoringGames}, we develop a variation of {\sc scoring konane}, with a bonus/penalty rule, for which you gladly would accept the offer, irrespective of any particular position.

\section{Games terminology}\label{sec:terminology}
Throughout, we assume best play, i.e., for example, `Left wins' means that Left can force a win against all of Right's strategies.
We first give the definitions common to all variants of additive combinatorial games.  Other concepts require the winning condition and these we give in separate sub-sections. We denote by $\mathbb{Np}$, the set of short, Normal-play games.

The word `game' has multiple meanings. Following \cite{AlberNW2007, Siege2013}, when referring to a set of rules,
the explicit game is in \textsc{small capitals},
otherwise, in proofs, `game' and `position' will be used interchangeably.

 Given a game $G$,
 an \textit{option} of $G$ is a position that can be reached in one move; a \emph{Left (Right)
 option}  of $G$ is a position that can be reached in one move by Left (Right).
 The set of Left, respectively Right, options of $G$ are denoted by $\GL $ and $\GR $ and
  we write $G^{\rm L}$ and $G^{\rm R}$ to denote a typical representative of $\GL $ and
  $\GR $ respectively. A combinatorial game is defined recursively as
   $G=\{\GL \mid\GR \}$. Further, $G$ is a \textit{short game} if it has a form
   $\{\GL \mid\GR \}$ such that $\GL $ and $\GR $ are finite
    sets of short games. A game $H$ is a \textit{follower} of a game $G$
 if there is any sequence of moves (including the empty sequence, and not necessarily alternating) starting at $G$
 that results in the game $H$.

The \textit{disjunctive sum} of the games $G$ and $H$ is the game
  $G+H$, in which a player may move in $G$ or in $H$, but not both.
  That is, $G+H = \{\GL +H, G+\HL\mid\GR +H, G+\HR\}$
  where, for example, if $\GL =\{G^{\rm L_1},G^{\rm L_2},\ldots\}$ then
  $\GL +H = \{G^{\rm L_1}+H,G^{\rm L_2}+H,\ldots\}$.
Clearly, the disjunctive sum operation is associative and commutative.

There is a well-known operation of `turning the board around',
that is, reversing the roles of both players.
In Normal-play games, given a game $G$, this new game is the additive inverse
of the old and the `turned' board is denoted by $-G$ giving rise
to the desirable statement $G-G=0$, indicating
that there is no advantage to either Left or Right in
$G-G$. In Scoring games and also in Mis\`ere-play, the underlying
structure is not necessarily a group, so for most games $G - G\ne 0$.
To avoid misleading equations, we call this operation \textit{conjugation}, and
denote the \textit{conjugate} by
 $\sim \! G$. If
$G=\{G^{\rm L_1},G^{\rm L_2},\ldots \mid G^{\rm R_1},G^{\rm R_2},\ldots \}$
then recursively
$\sim \! G = \{\sim \! G^{\rm R_1},\sim \! G^{\rm R_2},\ldots \mid \, \sim \! G^{\rm L_1},\sim \! G^{\rm L_2},\ldots\}$.

The notation $G=\{\GL\mid \GR\}$ has been identified with
Normal-play games in the literature. In this paper, since we will refer
to both Normal- and Scoring-play as different entities,
we will use $\langle \,\GL\mid \GR\, \rangle$ to refer to Scoring games
in order to avoid confusion.

The classes and subclasses of games that are mentioned in this paper, and in
papers about Mis\`ere-play games, all have some common
 properties and have been given a designation.
  \begin{definition} Let $\mathcal{U}$ be a set of combinatorial games.
Then $\mathcal{U}$ is a \emph{universe} if
\begin{itemize}
\item[(1)] $\mathcal{U}$ is closed under disjunctive sum;
\item[(2)] $\mathcal{U}$ is closed under taking options, that is, if $G\in\mathcal{U}$
then every $G^{\rm L}\in\GL $ and $G^{\rm R}\in\GR $ are in $\mathcal{U}$;
\item[(3)] $\mathcal{U}$ is closed under conjugation, that is, if $G\in\mathcal{U}$ then $\sim \! G\in\mathcal{U}$.
\end{itemize}
\end{definition}

\subsection{Scoring games terminology}\label{sec:ScoringGames}

All of the previous works on Scoring games used different terminology and notation
although the concepts were very similar. We unify the
notation; for example, even though players sometimes have a means of keeping score during the play,
the score is uniquely determined only when the player to move has no options.

\begin{definition}\label{basic}
Let $G$ be a game with no Left options. Then we write $\GL=\emptyset^{\ell}$
to indicate that, if Left to move, the game is over and the score is the real number $\ell$.
Similarly, if $\GR=\emptyset^r$, and it is Right's move, there are no Right options and the score is $r$.
We refer to $\emptyset^s$ as an \textit{atom} or, if needed for specificity, the $s$\textit{-atom}.
Scores will always be real numbers with the convention: Left wins if $s>0$;
Right wins if $s<0$ and the game is a tie (drawn) if $s=0$. Since
$\langle \, \emptyset^s\mid\emptyset^{s} \, \rangle$
results in a score of $s$ regardless of who moves next, we call this game $s$.
\end{definition}

 Games $\langle \, \emptyset^{\ell}\mid\emptyset^{r} \, \rangle$
with any of the three conditions $\ell<r$, $\ell=r$
and $\ell>r$ can occur in practice.
Allowing only $\ell=r$ gives the scoring universe studied in \cite{Stewa2011}.

Since addition and subtraction of real numbers does not pose a problem, we will revert
to using $-s$ instead of $\sim\! s=\langle \emptyset^{-s}\mid
\emptyset^{-s}\rangle$ for the conjugate of $s$.

\begin{example}
\textsc{diskonnect}\footnote{The first world championships, at
TRUe games May 2014, were played on a $8\times 8$ board with the middle
$2\times 2$ square empty. The authors placed 4th, 11th and 2nd respectively,
Paul Ottaway placed 1st and Svenja Huntemann 3rd.}
is played as \textsc{scoring-konane}, but with an additional bonus/penalty rule at the end: a
piece is \emph{insecure} if it can be captured by the opponent with a well-chosen
sequence of moves (ignoring the alternating-move condition), and otherwise,
the piece is \emph{safe}. When a player to move, say Left (Black), has no more
options then the game is over and Right (White)
removes all of Left's insecure stones on the board. This is the penalty a player
has for running out of options.
The score when the game ends is \textit{the number of stones Left has removed
minus the number of stones Right has removed.}
\end{example}

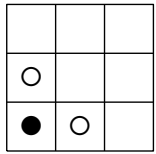
\begin{figure}[ht!]
\begin{center}
\scalebox{0.65}{
\psset{unit=1cm}
\begin{pspicture}(0,0)(3,4)
\psline(0,0)(0,3)(3,3)(3,0)(0,0)
\psline(1,0)(1,3)\psline(2,0)(2,3)
\psline(0,1)(3,1)\psline(0,2)(3,2)
\rput{0}(0.50,1.50){\scalebox{2.9}{$\circ$}}
\rput{0}(0.50,0.50){\scalebox{2.9}{$\bullet$}}
\rput{0}(1.50,0.50){\scalebox{2.9}{$\circ$}}
\end{pspicture}
}
\caption{An endgame in \textsc{konane}, \textsc{scoring-konane} or \textsc{diskonnect}.}\label{fig:disk1}
\end{center}
\end{figure}

The game in Figure \ref{fig:disk1} is  $1=\{ 0 \mid \emptyset \}$ in (Normal-play)
\textsc{konane},
$\langle\langle \emptyset^1\mid \emptyset^{1} \rangle \mid \emptyset^{0} \rangle =
\langle 1\mid \emptyset^{0}  \rangle $ in \textsc{scoring-konane}
and $\langle  1 \mid \emptyset^{2}  \rangle $ in \textsc{diskonnect}.

In \textsc{diskonnect}, if Left moves, she jumps one stone for a score of 1 and the game is over. If it is
Right to move, he has no moves and the bonus clause is invoked
and since there are two white stones that could be taken, the score is 2. Observe
that even though Left cannot, in play, actually take both stones, for each stone
 there is a legal sequence that leads to it being removed and so, each is insecure.

Scoring games  can be defined in a recursive manner using the atoms.
\begin{definition}\label{def:S}
The games born on Day 0 are
$S_0=\{\langle \, \emptyset^{\ell}\mid \emptyset^r \, \rangle : \ell,r\in \mathbb{R} \}$.
For $i=0,1,2,\ldots$, let $S_{i+1}$, the games born by Day $i+1$, be the set of games of the form
$\langle \,  {\cal G}\mid {\cal H} \, \rangle$, where ${\cal G}$ and
${\cal H} $ are non-empty finite subsets of $S_{i}$, or where either or both can be a single atom. The games in
$S_{i+1}\setminus S_{i}$ are said to have \textit{birthday}
$i+1$. Let $\mathbb{S}=\cup_{i\geqslant 0}S_i$.
\end{definition}

We now make explicit the effect of taking a disjunctive sum of Scoring games. (See Figure \ref{fig:gametree} for an example.)

\begin{definition}
Given two Scoring games $G$ and $H$, the disjunctive sum is given by:
\begin{eqnarray*}
G+H&=& \langle \, \emptyset^{\ell_1+\ell_2}\mid\emptyset^{r_1+r_2} \, \rangle, \quad\textrm{ if $G=\langle \, \emptyset^{\ell_1}\mid\emptyset^{r_1} \, \rangle$ and
$H=\langle \, \emptyset^{\ell_2}\mid\emptyset^{r_2} \, \rangle$;}\\
&=&\langle \, \emptyset^{\ell_1+\ell_2}\mid\GR +H,G+\HR \, \rangle, \textrm{ if
$G=\langle \, \emptyset^{\ell_1}\mid\GR  \, \rangle$  and
$H=\langle \, \emptyset^{\ell_2}\mid\HR \, \rangle$},\\
&{}&\qquad \textrm{ and at least one of $\GR $ and $\HR$
is not empty;}\\
&=&\langle \, \GL +H,G+\HL\mid \emptyset^{r_1+r_2} \, \rangle, \textrm{ if
$G=\langle \, \GL \mid\emptyset^{r_1} \, \rangle$  and
$H=\langle \, \GL \mid\emptyset^{r_2} \, \rangle$},\\
&{}&\qquad \textrm{ and at least one of $\GL $ and $\HL$
is not empty;}\\
&=&\langle \, \GL +H,G+\HL\mid\GR +H,G+\HR \, \rangle,
\textrm{ otherwise.}
\end{eqnarray*}
\end{definition}
Note that the option $\GL +H$ does not exist if $\GL $ is empty
(there is no addition rule for adding an empty set of options to a game).
For example,
consider $\langle \, \emptyset^1\mid 2 \, \rangle+ \langle \,  2\mid -1 \, \rangle$.
 If Left plays, she has no move in the first component, but does have a move,
so the score in the first component is not yet triggered. She must move to
 $\langle \,  \emptyset^1\mid 2 \, \rangle + 2$, whereupon Right moves to $2+2=4>0$,
and Left wins.
 If Right plays, he should move to $\langle \,   \emptyset^1\mid 2  \, \rangle + (-1)$;
 now Left has no move in the sum, and the score of Left's empty set of options is
triggered, giving a total
 score of $1-1=0$, a tie.
Note also that the addition of numbers is covered by
$p+s = \langle \, \emptyset^p\mid\emptyset^{p} \, \rangle +
\langle \, \emptyset^s\mid\emptyset^{s} \, \rangle=
\langle \, \emptyset^{p+s}\mid\emptyset^{p+s} \, \rangle$.

Game trees are a standard way to represent combinatorial games, and for Scoring games,
each leaf is typically labelled with a score of game; here, if one of the players run out of moves, the node
must have an atom attached to it (Figure~\ref{fig:gametree}). For Scoring (and Normal-play) games, we use the
convention that edges down and to the right represent a Right move, those down and
to the left represent a Left move.

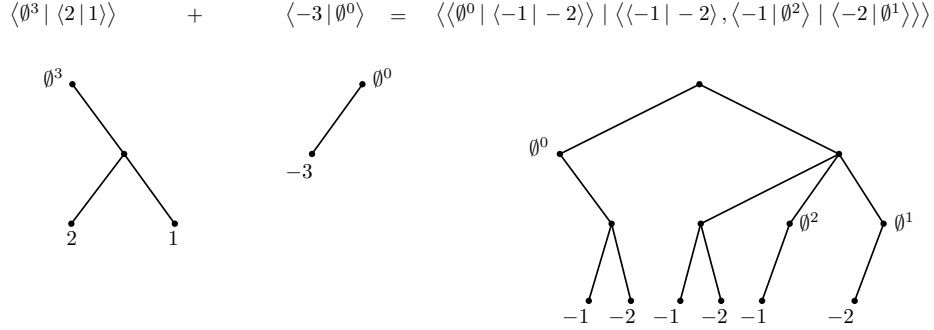
\begin{figure}[ht!]
\begin{center}
\scalebox{0.8}{
\psset{xunit=1.0cm,yunit=1.0cm,algebraic=true,dotstyle=o,dotsize=3pt 0,linewidth=0.8pt,arrowsize=3pt 2,arrowinset=0.25}
\begin{pspicture*}(-2.8,-1)(14,5)
\psline(-1.54,3.14)(-0.68,1.98)
\psline(-0.68,1.98)(-1.56,0.82)
\psline(-0.68,1.98)(0.16,0.82)
\rput[tl](-1.64,0.7){$2$}
\rput[tl](0.05,0.7){$1$}
\rput[tl](-2.02,3.4){$\emptyset^3$}
\psline(3.28,3.14)(2.44,1.98)
\rput[tl](2,1.8){$-3$}
\rput[tl](3.45,3.4){$\emptyset^0$}
\psline(8.88,3.14)(6.56,1.98)
\psline(8.88,3.14)(11.2,1.98)
\psline(6.56,1.98)(7.42,0.82)
\psline(7.42,0.82)(7.04,-0.46)
\psline(7.42,0.82)(7.74,-0.46)
\rput[tl](6.6,-0.6){$-1$}
\rput[tl](7.37,-0.6){$-2$}
\rput[tl](6.05,2.25){$\emptyset^0$}
\psline(11.2,1.98)(8.9,0.82)
\psline(8.9,0.82)(8.56,-0.46)
\psline(8.9,0.82)(9.24,-0.46)
\rput[tl](8.1,-0.6){$-1$}
\rput[tl](8.9,-0.6){$-2$}
\psline(11.2,1.98)(10.38,0.82)
\psline(10.38,0.82)(9.92,-0.46)
\rput[tl](9.55,-0.6){$-1$}
\rput[tl](10.54,1){$\emptyset^2$}
\psline(11.2,1.98)(11.94,0.82)
\psline(11.94,0.82)(11.46,-0.46)
\rput[tl](11,-0.6){$-2$}
\rput[tl](12.12,1){$\emptyset^1$}
\rput[tl](-2.58,4.5){$\left<\emptyset^3\,|\,\left<2\,|\,1\right> \right>$}
\rput[tl](2,4.5){$\left<-3\,|\,\emptyset^0\right>$}
\rput[tl](4.5,4.5){$\left< \left<\emptyset^0\,|\,\left<-1\,|\,-2\right> \right> \,|\,\left< \left<-1\,|\,-2\right>, \left<-1\,|\,\emptyset^2\right>\,|\, \left<-2\,|\,\emptyset^1\right> \right> \right>$}
\rput[tl](0.36,4.4){$+$}
\rput[tl](3.74,4.3){$=$}
\begin{scriptsize}
\psdots[dotstyle=*](-1.54,3.14)
\psdots[dotstyle=*](-0.68,1.98)
\psdots[dotstyle=*](-1.56,0.82)
\psdots[dotstyle=*](0.16,0.82)
\psdots[dotstyle=*](3.28,3.14)
\psdots[dotstyle=*](2.44,1.98)
\psdots[dotstyle=*](8.88,3.14)
\psdots[dotstyle=*](6.56,1.98)
\psdots[dotstyle=*](11.2,1.98)
\psdots[dotstyle=*](7.42,0.82)
\psdots[dotstyle=*](7.04,-0.46)
\psdots[dotstyle=*](7.74,-0.46)
\psdots[dotstyle=*](8.9,0.82)
\psdots[dotstyle=*](8.56,-0.46)
\psdots[dotstyle=*](9.24,-0.46)
\psdots[dotstyle=*](10.38,0.82)
\psdots[dotstyle=*](9.92,-0.46)
\psdots[dotstyle=*](11.94,0.82)
\psdots[dotstyle=*](11.46,-0.46)
\end{scriptsize}
\end{pspicture*}
}
\caption{The disjunctive sum of two game trees.}\label{fig:gametree}
\end{center}

\end{figure}

We now turn our attention to the partial order of Scoring games. We will use Left- and Right-scores, obtained from alternating play, for comparison of Scoring games.

   \begin{definition}\label{SStops}
    The \emph{Left-score} and the \emph{Right-score} of a scoring game $G$ are:
   \begin{eqnarray}
    \LS(G) &=& \begin{cases}r
                   & \text{if $\GL  =\emptyset^\ell$}, \\
                      \max(\RS(G^{\rm L})) & \text{otherwise};\end{cases} \\
    \RS(G) &=& \begin{cases}r
                   & \text{if $\GR =\emptyset^r$}, \\
                      \min(\LS(G^{\rm R})) & \text{otherwise}.\end{cases}
    \end{eqnarray}
    \end{definition}

    \begin{definition}\label{def:equality}(Inequality in Scoring Universes)\\
Let $\mathcal{U}\subseteq \mathbb{S}$ be a universe of combinatorial Scoring games,
and let $G, H\in {\cal U}$. Then
    $G\geqslant  H$ if $$\LS(G+X)\geqslant  \LS(H+X)$$ and $$\RS(G+X)\geqslant  \RS(H+X),$$
for all $X\in \mathcal{U}$.
    \end{definition}
     For game equivalence, we replace all inequalities in the definition, by equalities. It follows that any universe of Scoring games
$\mathcal{U}\subseteq \mathbb{S}$ is a monoid, that is
$0 + X = X$ for any $X\in {\cal U}$.

       \section{A natural scoring universe}\label{sec:obstacle}
Normal-play games can be regarded as Scoring games,
with all scores being zero. From a mathematical perspective,
it is of  interest to know when they can be embedded in a scoring universe in an order preserving way.

\begin{definition}\label{defn:naturalembed}
Let $\mathcal{U}\subseteq \mathbb{S}$ be a universe of Scoring games
with $\mathbb{R}\subset \mathcal U$.
Define the \emph{Normal-play mapping} $\zeta:\mathbb{Np} \hookrightarrow \mathcal{U}$ as
$\zeta(G)=\Num{G}\in \mathcal{U}$ where $\Num{G}$
is the game obtained by replacing each empty set of options in the
followers of $G\in \mathbb{Np}$, with the $0$-atom,~$\emptyset^0$.
\end{definition}

Since each atom in $\Num{G}$ is the $0$-atom, the
outcome is obviously a tie when the game is played in isolation.
But the importance of the $\mathbb{Np}$-mapping is revealed in Definition \ref{def:natural}, and
later in Theorem~\ref{thm:emb}.
The mapping $f: X\rightarrow Y$ is an \emph{order-embedding} if, for all $x_1,x_2\in X$, $x_1\leqslant x_2$ if and only if $f(x_1)\leqslant f(x_2)$.

\begin{definition}\label{def:natural}
A scoring universe $\mathcal{U}$ is \emph{natural} if the $\mathbb{Np}$-mapping $\zeta$ is an order-embedding.
\end{definition}

Tactically, the games $\Num{n}$ can be regarded as \textit{waiting-moves}, for $n$ an integer.
 We will say that Left (Right) \textit{waits} to mean that she uses
one of the waiting-moves in $G + \Num{n}$ if $n$ is positive (negative).

For example, in the disjunctive sum $\Num{2}+\langle \,  -4\mid\langle \, -3  \mid 5  \, \rangle \, \rangle$
Left is happy to play her waiting-move giving the option
$\Num{1}+\langle \, -4\mid\langle \,  -3\, |\, 5  \, \rangle  \, \rangle$; Right responds to
$\Num{1}+\langle \,  -3 \mid5  \, \rangle$; Left again waits giving
$\Num{0}+\langle \,  -3 \mid5  \, \rangle = \langle \,  -3 \mid5  \, \rangle$; Right is forced to move to $5$.
Further analysis shows that Left much prefers
$\Num{2}+\langle \,  -4\mid\langle \,  -3 \mid5  \, \rangle  \, \rangle$ to
$\Num{1}+\langle \,  -4\mid\langle \,  -3 \mid5  \, \rangle  \, \rangle$.

However, there is an even more compelling reason coming from tactical situations
that appear naturally
within the games we play. We
invite the reader, playing Left, to contemplate which of the two games $G_1$ and
$G_2$, in Figure~\ref{fig:disk4} they would prefer to add to a disjunctive sum of Scoring games.

\begin{figure}[ht!]
\begin{center}
\scalebox{0.65}{
\psset{unit=1cm}
\begin{pspicture}(0,0)(3,4)
\psline(0,0)(0,3)(5,3)(5,0)(0,0)
\psline(1,0)(1,3)\psline(2,0)(2,3)\psline(3,0)(3,3)\psline(4,0)(4,3)
\psline(0,1)(5,1)\psline(0,2)(5,2)
\rput{0}(0.50,0.50){\scalebox{2.9}{$\bullet$}}\rput{0}(1.50,0.50){\scalebox{2.9}{$\circ$}}\rput{0}(3.50,0.50){\scalebox{2.9}{$\circ$}}
\end{pspicture}
\hspace{5 cm}
\begin{pspicture}(0,0)(3,4)
\psline(0,0)(0,3)(3,3)(3,0)(0,0)
\psline(1,0)(1,3)\psline(2,0)(2,3)
\psline(0,1)(3,1)\psline(0,2)(3,2)
\rput{0}(2.50,1.50){\scalebox{2.9}{$\circ$}}
\rput{0}(0.50,0.50){\scalebox{2.9}{$\bullet$}}\rput{0}(1.50,0.50){\scalebox{2.9}{$\circ$}}
\end{pspicture}
}
\caption{The \textsc{Diskonnect}-games $G_1$ and $G_2$.}\label{fig:disk4}
\end{center}
\end{figure}
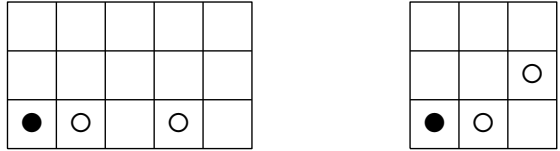

From a Normal-play point of view, it is natural that more waiting-moves
give a tactical advantage over fewer and so the order-embedding
of Normal-play games is something to be desired if
we wish to use any Normal-play intuition in Scoring play. More concretely,
in $G_1$, Figure~\ref{fig:disk4}, by `general principles',
Left's jumping only one stone is clearly dominated
(we invite the reader to consider the intuition of this) so
\[ G_1 = \langle \, 2\mid \emptyset^2  \, \rangle = 2 + \langle \, 0\mid \emptyset^0  \, \rangle =
2+\langle \, \langle \,  \emptyset^0\mid \emptyset^0  \, \rangle \mid \emptyset^0  \, \rangle = 2 + \Num{1}
\]
and
\[G_2 = \langle \, 1+\langle \,  1\mid \emptyset^1 \, \rangle \mid \emptyset^2  \, \rangle
= 2 + \langle \, \langle \, \langle \,  \emptyset^0\mid \emptyset^0
\, \rangle \mid \emptyset^0  \, \rangle \mid \emptyset^0  \, \rangle = 2+\Num{2}
\]
Here, $G_2$, with the extra waiting-move, seems preferable
over $G_1$.

Games of the form $\langle \,  \emptyset^\ell\mid\GR  \, \rangle $
and $\langle \, \GL \mid\emptyset^r \, \rangle $ (including $\langle \,  \emptyset^\ell\mid  \emptyset^r\, \rangle $)
will be called \textit{atomic} in general, or
\textit{Left-atomic} and \textit{Right-atomic}, respectively, if more precision is required.

\begin{definition}\label{def:guaranteed}
Let $G\in \mathbb{S}$ be an atomic game. Then $G$ is \emph{stable} if $\LS(G)\leqslant \RS(G)$, and $G$ is guaranteed if, for every atom $\emptyset^s$ which is a follower of $\GL$, and every atom $\emptyset^t$ which is a follower of $\GR$, $s\leqslant t$. In general, a game in $\mathbb{S}$ is \textit{stable} (\emph{guaranteed}) if every atomic follower is stable (guaranteed). Let $\GB$ be the class of all guaranteed Scoring games.
\end{definition}

It is clear that $\GB$
 is a universe, since `guaranteed' is an hereditary property;
 moreover, this class is also generated recursively when $\mathbb{S}$ is generated,
see Definition~\ref{def:S}. Note that $G\in \GB$ implies that $G$ is stable.

To motivate our approach for guaranteed games,
we first prove an intermediate result on the stable
games. In particular, the existence of `hot-atomic-games', such as $\langle \,  \emptyset^3\mid-2 \, \rangle$, prevents $\zeta$ from being
an order-embedding.

\begin{theorem}\label{thm:sbU}
If $\mathcal{U}$ is a natural universe, then each game in $\mathcal{U}$ is stable.
\end{theorem}

\begin{proof}
Suppose that $\mathcal{U}$ contains a non-stable game $G = \langle \, \emptyset^\ell \mid \GR  \, \rangle\in\mathcal{U}$, where $\ell > \RS(G)$.
Choose $k$ where $\ell>k>\RS(G)$ and put $X=G-k$. Consider the Left-score of $0+X$:
 \begin{eqnarray*}
 \LS(0+X) &=& \LS(\langle \, \emptyset^\ell\mid\GR \, \rangle-k),\quad \\
 &=&\ell-k>0, \mbox{since Left has no move};
 \end{eqnarray*}
 and the Left-score of $\Num{1} +X$:
  \begin{eqnarray*}
  \LS(\,  \Num{1}+X) &=& \LS( \, \Num{1}+\langle \, \emptyset^\ell\mid\GR  \, \rangle-k),\quad \\
  &=&\RS(\langle \, \emptyset^\ell\mid\GR  \, \rangle)-k, \mbox{since Left only had the waiting-move},\\
  &=&\RS(G)-k < 0.
  \end{eqnarray*}
 Since $\LS(0+X)>\LS( \, \Num{1}+X)$, it follows from
 Definition \ref{def:equality}
  that
 $0$ is not less than $\Num{1}$ and consequently $\mathcal{U}$
 is not natural.
\end{proof}

Universes in which every game is stable are not necessarily natural.

\begin{example}\label{ex:non-natural}
For example,  in $\mathbb{Np}$ the game 0 has many forms,
and in a natural universe they all are mapped to the same game.
However, in $\mathbb{Np}$, let $*=\{0\mid 0\}$, which gives $0=\{*\mid*\}$.
Consider the Scoring-game
$G=\langle \, \emptyset^2\mid\langle \, \langle \, -5\mid5\, \rangle\mid-5 \, \rangle \, \rangle$, which
is stable because $\LS(G)=2<\RS(G)=5$.
Now $\LS( \, \Num{0}+G)>0$ and $\LS(  \Num{\{*\mid*\}}+G)<0$,
i.e. $\Num{\{*\mid*\}}\ne \Num{0} = 0$ and
the two representations of 0
are mapped to different games.
\end{example}

We already know that $\GB$ is a stable universe, and the non-existence of hot-atomic-games ensures that it is also natural.

\begin{theorem}\label{thm:emb}
The universe $\GB$ is natural.
\end{theorem}

\begin{proof}
We must demonstrate that the Normal-play mapping is an order-embedding.
First we show that $\zeta$ is order-preserving.

Consider $G, H\in\mathbb{Np}$ such that $G\geqslant  H$.
We want to argue that $\Num{G}\geqslant  \Num{H}$ in $\GB$. The next arguments
 show that $\LS(\, \Num{G}+X)\geqslant  \LS(\, \Num{H}+X)$,
for all $X\in\GB$.
To this purpose, we induct on the sum of the birthdays of $G$, $H$ and $X$.

 First, suppose that Left has no options in $\Num{H}+X$. This occurs when
  $X=\langle \,\emptyset^x\mid\XR\, \rangle$ and
$H=\{\emptyset | \,\HR\}$
(that is $\Num{H}=\langle \,\emptyset^0\mid\Num{H}^{\cal R}\, \rangle$).
In this case, $\LS(\, \Num{H}+X)=x$. Now
$\Num{G}+X=\Num{G}+\langle \,\emptyset^x\mid\XR\, \rangle$.
If Left has a move in $G$ then, in $X$, which is a guaranteed game,
the score will be larger than or equal to $x$. If Left cannot play in $G$ then
the game is over and she has a score of $x$.
Therefore, in both cases, because the score in $\Num{G}$ is always
$0$, $\LS(\, \Num{G}+X)\geqslant  x$.

Now we assume that Left has a move in $\Num{H}+X$.
If there is a Left move, $X^{\rm L}$, such that
  $\LS(\, \Num{H}+X)=\RS(\, \Num{H}+X^{\rm L})$, then, by induction
  $\RS(\, \Num{H}+X^{\rm L})\leqslant  \RS(\, \Num{G}+X^{\rm L})$.
  By the definitions of the Left- and Right-scores,
  $\RS(\, \Num{G}+X^{\rm L})\leqslant  \LS(\, \Num{G}+X)$, giving
$\LS(\, \Num{H}+X) \leqslant  \LS(\, \Num{G}+X)$. The remaining case is that there is a Left move in $\Num{H}$ with
$\LS(\, \Num{H}+X)=\RS(\, \Num{H}^{\rm L}+X)$. In $\mathbb{Np}$,
$G\geqslant  H$, i.e., $G-H\geqslant  0$ and  so Left has a winning move
in $G-H^{\rm L}$. There are two possibilities, either $G^{\rm L}-H^{\rm L}\geqslant  0$ or
 $G-H^{\rm LR}\geqslant  0$. If the first occurs, then $G^{\rm L}\geqslant  H^{\rm L}$
and, by induction, $\RS(\, \Num{G}^{\rm L}+X)\geqslant  \RS(\, \Num{H}^{\rm L}+X)$
which gives us the inequalities
\[\LS(\, \Num{H}+X)=\RS(\, \Num{H}^{\rm L}+X)\leqslant
\RS(\, \Num{G}^{\rm L}+X)\leqslant \LS(\, \Num{G}+X).
\]
If $G-H^{\rm LR}\geqslant  0$ occurs, then, by induction,
$\LS(\, \Num{G}+X) \geqslant  \LS(\, \Num{H}^{\rm LR}+X)$. By the definitions of Left- and Right scores,
we also have $\LS(\, \Num{H}^{\rm LR}+X) \geqslant \RS(\, \Num{H}^{\rm L}+X)$ and
since we are assuming that $\LS(\, \Num{H}+X)=\RS(\, \Num{H}^{\rm L}+X)$,
we can conclude that $\LS(\, \Num{G}+X)\geqslant  \LS(\, \Num{H}+X)$.

The arguments to show that $\RS(\, \Num{G}+X)\geqslant  \RS(\, \Num{H}+X)$
for all $X\in\mathbb{\GB}$ are analogous and we omit these.
Since $\LS(\, \Num{G}+X)\geqslant  \LS(\, \Num{H}+X)$ and
$\RS(\, \Num{G}+X)\geqslant  \RS(\, \Num{H}+X)$ for all $X\in\mathbb{\GB}$, then, by Definition~\ref{def:equality},
$\Num{G}\geqslant  \Num{H}$.

To demonstrate that $\zeta$ is an order-embedding, it suffices to show that $G>H$ implies $\Num{G}>\Num{H}$ and
$G \mid\mid H$ implies $\Num{G} \mid\mid \Num{H}$.
We already know that $G>H$ implies $\Num{G}\geqslant \Num{H}$, so it suffices to show that $\Num{G}\neq \Num{H}$. Consider the distinguishing game $X=\; \sim\! \Num{H} +\langle \, -1\mid 1\, \rangle$. We get $\Num{H} + X=H+ \sim \!\! H +\langle \, -1\mid 1\, \rangle$, and Left (next player) loses playing first in $H-H$. So $\LS(\Num{H} + X)=-1$ (no player will play in the Zugzwang). But, similarly, $G>H$ implies $\LS(\Num{G} + X)=1$, since Left wins playing first in $G-H$. Hence $\Num{H}\ne \Num{G}$.

To prove that $H\mid\mid G$, we use the same distinguishing game $X=\; \sim \! \Num{H} +\langle \, -1\mid~1\, \rangle$.
We have that $-1=\LS(H+ X)<\LS(G+ X)=1$ and $1=\RS(H+ X) >\RS(G+ X)=-1$, which proves the claim.
\end{proof}

Some of the properties of Normal-play will hold in our Scoring universe,
 but there is some cost to play in a fairly general Scoring universe; there
are non-invertible elements (see also Section \ref{sec:survey}). Hence
comparison of games, in general, cannot be carried out as easily as in
 Normal-play, where $G\geqslant H$ is equivalent to $G-H\geqslant 0$ (Left win
playing second in $G-H$). Here we begin by demonstrating how to
compare games with scores.

Note that since $\Num{-n} =\, \sim\!\Num{n}$ we will revert to the less cumbersome
notation $-\Num{n}$.

    \begin{definition}
    Let $G$ be a game in $\mathbb{S}$. Then
    $\LSu(G)=\min\{\LS(G-\Num{n}\, ):n\in\mathbb{N}_0\}$ is the
    \emph{Right's pass-allowed Left-score} (\emph{pass-allowed Left-score}). The \emph{pass-allowed Right-score} is defined analogously,
    $\RSo(G)=\max\{\RS(G+\Num{n}):n\in\mathbb{N}_0\}$.
    \end{definition}
In brief, the `overline' indicates that Left can pass and the `underline' that Right can pass.
    \begin{lemma}
    Let $G\in \mathbb{S}$. Then
\begin{enumerate}
\item $\LS(G)\geqslant  \LSu(G)$ and $\RS(G)\leqslant  \RSo(G)$;
\item $\RSo(G+H)\geqslant \RSo(G)+\RSo(H)$
and
$\LSu(G+H)\leqslant \LSu(G)+\LSu(H)$.
\end{enumerate}
    \end{lemma}

    \begin{proof}
    The inequalities in statement 1 are obvious from the definition of the $\min$-function,
and since $G+\Num{0} =G$.

If Left answers Right in the same component (including the possibility of a waiting-move) then this is not the full complement
 of strategies available to her, and this proves that
$ \RSo(G+H)\geqslant \RSo(G)+\RSo(H)$.

The inequalities for $\LSu$ are proved similarly.
    \end{proof}

    \begin{definition}
    Let $\ell \in \mathbb{R}$.
    Then, $G$  is \textit{left-$\ell$-protected} if
     $\LSu(G)\geqslant  \ell$ and, for all
     $G^{\rm R}$, there exists $G^{\rm RL}$
     such that $G^{\rm RL}$ is left-$\ell$-protected. Similarly,
$G$ is Right-$r$-protected if
     $\RSo(G)\geqslant  r$ and, for all
     $G^{\rm L}$, there exists $G^{\rm LR}$
     such that $G^{\rm LR}$ is right-$r$-protected.
    \end{definition}

  The concept of $\ell$-protection allows for comparisons with numbers in $\GB$.

\begin{theorem}\label{thm:comp}
Let $G\in \GB$. Then $G\geqslant  \ell$ if and only if $G$ is left-$\ell$-protected.
\end{theorem}
\begin{proof}

\smallskip
($\Rightarrow$) Suppose that $G$ is not left-$\ell$-protected.

We will use distinguishing games of the form
$X=\langle \,  \emptyset^a\mid b-  \Num{n}  \, \rangle$
to obtain contradictory inequalities $\RS(G+X)<0<\RS(\ell+X)$.
The cases $1$ and $2$ are general considerations that don't need induction and that can be
used whenever we want. We begin with the base case of the induction on the rank of $G$,
that will be used to prove case $3$.\\

\noindent Base case (rank 0): $G=\langle \, \emptyset^v\mid \emptyset^s\, \rangle$, with $v\leqslant s$.\\

Because $G$ is not left-$\ell$-protected, we have $\LS (G) = v < \ell$. Therefore, in order to have a
distinguishing game to build the induction, we consider
$X=\langle \,  \emptyset^a\mid  a+\Num{0}  \, \rangle$ where $-\ell<a<-v$. Then, $\RS(G+X)=v+a<0<a+\ell=\RS(r+X)$.\\

\noindent {\rm Case 1:} $\LSu(G)=v<\ell$.\\

Consider $X=\langle \,  \emptyset^a\mid a-   \Num{n}  \, \rangle$,
where $-\ell<a<-v$, and where $n$ is large enough to obtain $\LSu(G)$. Then,
$$\RS(G+X)\leqslant v+a<0$$
$$\RS(\ell+X)=\ell+a>0$$

This is contradictory with $G\geqslant \ell$.\\

\noindent Case 2: There exists a $G^{\rm R}$ that is Left-atomic.\\

If $G^{\rm R}=\langle \,  \emptyset^v \mid (G^{\rm R})^\mathcal{R}  \, \rangle$,
then consider $X=\langle \,  \emptyset^a\mid b+\Num{0}  \, \rangle$ such that
$a<-\ell$, $b>-\ell$ and $a<b$. Then,
$$\RS(G+X)\leqslant \LS(G^{\rm R}+X)=v+a<0$$
$$\RS(\ell+X)=\ell+b>0.$$

\noindent Case 3: There exists $G^{\rm R}$ such that
${\cal G}^{\rm RL}\neq\emptyset$ and all $G^{\rm RL}$ are not left-$\ell$-protected.\\

Consider $G^{\rm RL_1}$, $G^{\rm RL_2}$, $G^{\rm RL_3}$,\ldots,$G^{\rm RL_k}$;
the games in ${\cal G}^{\rm RL}$. By induction, for all $i\in \{1,\ldots k\}$,
consider the distinguishing games
  $X_i = \langle \,  \emptyset^{a_i}\mid b_i- \Num{n_i}\, \rangle$
such that
  $$\RS(G^{\rm RL}_i+X_i)<0<\RS(\ell+X_i),$$ for all $i$.

Let $X=\langle \,   \emptyset^{  \min(a_i)}\mid \min(b_i)- \max(\Num{n_i})  \, \rangle$.
By convention and Theorem~\ref{thm:emb}, each $X_i$ is no better than $X$ for Right.
Hence, we get that $\RS(G^{\rm RL}_i+X)<0$, for all $i$. We also get $\RS(\ell+X)=\ell+\min(b_i)>0$.

Therefore, $\RS(G+X)\leqslant \LS(G^{\rm R}+X)=\max_i \RS(G^{\rm RL}_i+X) <0$.
This contradicts $G\geqslant \ell$, since $\RS(\ell+X)>0$, and thus finishes the induction step.\\

\smallskip
($\Leftarrow$) Assume that $G$ is left-$\ell$-protected.
We need to prove that $\RS(G+X)\geqslant  \RS(\ell+X)$ and $\LS(G+X)\geqslant
\LS(\ell+X)$ $\forall X\in\GB$. Fix some $X\in\GB$
and we proceed by induction on the sum of the ranks of $G$ and $X$.

Suppose there is no Right option in $G+X$, that is,
$G=\langle \GL  \mid\emptyset^p \rangle $ and $X=\langle \XL \mid\emptyset^q \rangle$.
It follows that $\RS(G+X) = p+q$ and $\RS(\ell+X) = \ell+q$.
Let $s=\LSu(G)$. Since $G$ is left-$\ell$-protected, we get $s\geqslant  \ell$. Further, since $G$ is
guaranteed, we have that $s\leqslant p$, and therefore
$\RS(G+X) = p+q \geqslant  s+q \geqslant  \ell+q =\RS(\ell+X)$.

We may now suppose that Right has a move in $G+X$.
Suppose that Right's best move is in $X$, to say $G+X^{\rm R}$. We then have the chain of relations
\begin{eqnarray*}
\RS(G+X) &=& \LS(G+X^{\rm R}), \mbox{ by assumption},\\
&\geqslant & \LS(\ell+X^{\rm R}), \mbox{ by induction on the sum of the ranks},\\
&= & \RS(\ell+X), \mbox{ by definition and since numbers have empty sets of options.}
\end{eqnarray*}
If Right's best move is in $G$, to say $G^{\rm R}+X$, then we have the relations
\begin{eqnarray*}
\RS(G+X) &=& \LS(G^{\rm R}+X), \mbox{ by assumption,}\\
&\geqslant & \RS(G^{\rm RL}+X), \mbox{ Left might have chosen a non-optimal option,}\\
&\geqslant & \RS(\ell+X), \mbox{ by induction, since $G^{\rm RL}$ is left-$\ell$-protected.}
\end{eqnarray*}

In all cases we have that $\RS(G+X)\geqslant  \RS(\ell+X)$.

\smallskip
If Left has a move in $X$ then let $X^{\rm L}$ be the move such that $\LS(\ell+X)=\RS(\ell+X^{\rm L})$.
By the previous paragraphs, we have that $\RS(\ell+X^{\rm L})\leqslant  \RS(G+X^{\rm L})$.
Since moving to $X^{\rm L}$ may not be the best Left move in $G+X$,
we know that $\RS(G+X^{\rm L})\leqslant  \LS(G+X)$,
i.e. $\LS(\ell+X)\leqslant  \LS(G+X)$.

Now, if Left has no move in $X$ then $X = \langle \,  \emptyset^s|\XR \, \rangle$.
We know that $\LS(\ell+X) = \ell + s$. In $G+X$ we restrict Left's strategy, which can only give her the same or a
lower score. We denote the continuation of the two games as the $G$ and $X$ components.
Whenever there is a Left option in the $G$ component, Left plays the best one.
If there is no option in the $G$ component, then and only then Left plays in the $X$ component. Any move by Right in the
$X$ component is a waiting-move in the $G$ component so
Left achieves a score of at least $\LSu(G)$; moreover, by assumption,
$\LSu(G)\geqslant  \ell$. Since $X$ is guaranteed, the score from this component
is at least $s$. That is, $\LS(G+X)\geqslant \ell+s = \LS(\ell+X)$.

For all $X$, we have shown that $\LS(G+X)\geqslant  \LS(\ell+X)$
and $\RS(G+X)\geqslant  \RS(\ell+X),$ that is, $G\geqslant  \ell$.
\end{proof}

An interesting question  is: \textit{What games, and how many, are equal to 0?}
We have, for example, $G=\langle \,   \langle \,   1\mid 0  \, \rangle
 \mid  \langle \,   0\mid -1  \, \rangle   \, \rangle  = 0$, a game not obtained from the natural embedding.
But each natural embedding of Normal-play game equal to zero maps to zero in $\GB$.
Note that, for example, $\langle \,   0\mid 0  \, \rangle \ne 0$, which is of course what we want
since in $\mathbb{Np}$, $\cgstar = \{  0\mid 0  \}$ (compare this with Milnor's scoring universe
in Section \ref{sec:survey}). But also
$ \langle \,    \langle \,   1\mid 0  \, \rangle   \mid \emptyset^0   \, \rangle = 0$,
a game that is not obtained from a 0 in Normal-play and neither is it a dicot.

The size of the equivalence class of $0$ gives a lower bound on the sizes of the other equivalence classes
since $G+H=H$ if $G=0$.
Further, note that, if $G\geqslant 0$, then $G$ is Left-0-protected, and thus Right loses moving first, which is similar to the situation in Normal-play.
The following corollary of Theorem \ref{thm:comp} gives a criteria that enables us to determine when a game is equal to 0.

\begin{corollary}\label{0corollary}
Let $G\in \GB$. Then $G=0$ iff G is left-$0$- and right-$0$-protected,
that is iff $\LSu(G) = \RSo(G)= 0$ and, for all $G^{\rm R}\in \GR $, there exists
$G^{\rm RL}\geqslant  0$ and, for all $G^{\rm L}\in \GL $, there exists
$G^{\rm LR}\leqslant  0$.
\end{corollary}

%

\section{Survey of other scoring universes}\label{sec:survey}
Let us begin by listing some of the game properties of the
other scoring universes in the literature. In the
columns, we list the authors of the three most relevant scoring universes in
relation to this work; the properties are discussed, as appropriate, later
in this section. The `Yes' in Stewart's column means that
there are no known `practical methods', or that the answer is trivial as for
the invertible elements. The games in Milnor's and Ettinger's universes are
trivially stable, since the only atomic games are numbers.
\medskip
Summarizing:

\begin{center}
\begin{tabular}{|c|c|c|c|}
\hline
{\bf Properties}& {\bf Milnor}&{\bf Ettinger}&{\bf Stewart}\\
\hline Ordered Abelian Group&Yes&No&No\\
 \hline Equivalence Class of 0  & Large & Large & Small\\
\hline  Invertible Elements &All& Many&Numbers \\
\hline Constructive Comparisons & Yes&Yes&No \\
\hline   Greediness Principle & Yes&Yes&No \\
 \hline Game Reductions & Yes &Yes& `Yes'\\
\hline   Characterization of Invertible Elements & Yes&Yes&`Yes'\\
\hline   Unique Canonical Forms & Yes&No&`Yes' \\
\hline  All Games Stable & `Yes' &`Yes' &No\\
\hline  Natural Embedding&No&No&No\\
\hline
\end{tabular}
\end{center}

\subsection{Milnor's non-negative incentive games }

Milnor \cite{Milno1953} and Hannor \cite{Hanne1959} considered dicot Scoring games in which
there is a non-negative incentive for each player to move,
and where the atoms (base of recursion) are real numbers.
We denote this universe by $\mathbb{PS}$. Non-negative incentive translates to $\LS(G) \geqslant \RS(G)$, for all positions, that is
\emph{Zugzwang} situations never occur and the universe is an abelian group.
As soon as $\LS(G)=\RS(G)$, the game is over and the players add up the score.

There is a similarity between these games
and Normal-play games. If the scores are always integers, at the end, the players
could count the score by imagining they are making `score'-many independent moves.
Using this idea, there have been many advances in the endgame of
\textsc{go} (last point). Games such as \textsc{amazons}\footnote{see
\url{http://en.wikipedia.org/wiki/Game_of_the_Amazons}}
and \textsc{domineering} can also be thought of as `territorial games' (where the score depends on the size of captured land),
but, in the literature, they have so far been analyzed under the guise of Normal-play.

Because the games in this universe have non-negative incentive
and all games of $\mathbb{Np}$ would have 0 incentive (no change in score)
there is no natural embedding of $\mathbb{Np}$ into $\mathbb{PS}$.
In Milnor's universe when we have the disjunctive sum of a Normal-play component with a non-negative incentive component,
the outcome is determined by the non-negative incentive component. Both players want to play in the non-negative incentive component because there are no Zugzwangs.
A Normal-play component alone is a tie. So, any Normal-play component is irrelevant; i.e.,
all embedded Normal-play components are equal to zero, and so, $\mathbb{PS}$ is not natural.

\subsection{Ettinger's dicot Scoring games}

Ettinger's universe \cite{Ettin1996,Ettin2000} also consists of dicot games, but
\emph{Zugzwang} games like $\langle \, -3\mid3\, \rangle $ are now allowed.
In Ettinger's universe,  $\mathbb{DS}$, the atomic games are the real numbers
 and moreover, there are no games of the form $\langle \,  \emptyset^r\mid \emptyset^s  \, \rangle, r\neq s$ (although the latter game is a dicot).
 He noted that real numbers are not necessary, just some values taken from an ordered abelian group.

The definitions of (in)equality and of Left- and Right-scores are as in Definitions~\ref{SStops} and \ref{def:equality}.

\begin{definition}
The disjunctive sum $G+H$ with $G,H\in \mathbb{DS}$ reduces to
\begin{displaymath}
G+H= \left\{ \begin{array}{ll}
r+s & \textrm{if $G=r$ and $H=s$ are numbers},\\
\langle \,   \GL +H,G+\HL\mid \GL +H, G+\HR  \, \rangle & \textrm{otherwise.}
\end{array} \right.
\end{displaymath}
\end{definition}

Disjunctive sum is associative and commutative, and the universe is partially ordered.
The problem is that games do not necessarily have inverses.
    The following concept was used by Ettinger \cite[p. 20-22]{Ettin1996}.

    \begin{definition}
    Let $r\in \mathbb{R}$. Then, $G$  is \textit{left-$r$-safe}
    if $\LS(G)\geqslant  r$ and, for all $G^{\rm R}$, there exists $G^{\rm RL}$ such that
    $G^{\rm RL}$ is left-$r$-safe. The concept of right-$r$-safety is defined analogously.
    \end{definition}

Two important results are:

\begin{theorem}[Ettinger's Theorem]\label{Ettingertheorem}
Let $r\in \mathbb{R}$ and $G\in \mathbb{DS}$.
Then, $G\geqslant  r$ iff $G$ is left-$r$-safe.
\end{theorem}

Writing the explicit condition for `safety', we get:

\begin{corollary}[Ettinger's Corollary]\label{Ettingercorollary}
Let $G\in \mathbb{DS}$. Then $G=0$ iff $\LS(G)=\RS(G)=0$ and for all
$G^{\rm R}\in \GR $ there exists $G^{\rm RL}\geqslant  0$ and for all
$G^{\rm L}\in \GL $ there exists $G^{\rm LR}\leqslant  0$.
\end{corollary}

In \cite{Ettin1996} (p. 48), it is proved that $G\in \mathbb{DS}$ is invertible iff $G +\sim \! G = 0$.
In fact, Ettinger proved that there are non-invertible elements and
$\mathbb{DS}$ is just a semigroup (monoid). Namely, consider
$G=\langle \,  \langle \,  1 \mid -1  \, \rangle \mid \langle \,  1 \mid 1  \, \rangle  \, \rangle $.
 Then $\sim \! G = \langle \, \langle \, -1\mid -1\, \rangle \mid \langle \,
1\mid -1 \, \rangle \, \rangle $. Now $\LS(G + \sim \! G)=-2$. Therefore,
using the distinguishing game $0$, $\LS(G + \sim \! G+0) < \LS(0+0)$ and, so,
by Definition~\ref{def:equality}, $G + \sim \! G\neq 0$.

Lacking a group structure, how do we, constructively, know if $G\geqslant  r$?
As explained in Theorem \ref{Ettingertheorem}, Ettinger solved the problem with the concept of $r$-safety.

The Greediness  principle holds in $\mathbb{DS}$. Also, $\mathbb{DS}$
does not have hot-atomic-games;
the only games with empty sets of options are the real numbers.

Also, in \cite{Ettin1996}, despite there being reductions (domination and reversibility)
this does not lead to a unique canonical form for the equivalence classes.

Finally, since $\mathbb{DS}$ is a dicot universe and $\mathbb{Np}$ is not
($\{ 0 \mid \emptyset \}$ is a game in $\mathbb{Np}$)
there is no natural order-preserving embedding from $\mathbb{Np}$
into $\mathbb{DS}$. The only relation is with the subgroup of
$\mathbb{Np}$ constituted by the dicot games (formerly called \textit{all-small}) (\cite{AlberNW2007} p.185-).

\begin{figure}[h]
\begin{center}
\scalebox{0.66}{
\psset{xunit=1.0cm,yunit=1.0cm,algebraic=true,dotstyle=o,dotsize=3pt 0,linewidth=0.8pt,arrowsize=3pt 2,arrowinset=0.25}
\begin{pspicture*}(-3.5,-7)(10.8,5.5)
\pscircle[linewidth=2pt](-0.8,2.22){2.56}
\pscircle[linewidth=2pt](-0.8,2.22){0.9}
\rput[tl](-1.1,5.3){$\mathbb{Np}$}
\rput[tl](-1.62,3.84){$Dicot\,\,\mathbb{Np}$}
\pscircle[linewidth=2pt](7.9,2.02){2.56}
\rput[tl](7.72,5){$\mathbb{DS}$}
\pscircle[linewidth=2pt](3.66,-3.92){2.56}
\pscircle[linewidth=2pt](3.66,-3.92){0.9}
\rput[tl](3.36,-0.62){$\mathbb{GS}$}
\rput[tl](3.1,-2.6){$Dicots$}
\psline{->}(0.1,2.2)(5.34,2.18)
\rput[tl](1.95,2.94){$Order\,\,Embedding$}
\psline{->}(0.47,-0.01)(1.97,-2)
\rput[tl](-2.4,-0.66){$Order\,\,Embedding$}
\psline{->}(6.63,-0.2)(4.31,-3.29)
\rput[tl](6.3,-1.06){$Not\,\,Order\,\,Embedding$}
\rput[tl](7.2,2.1){$Theorem\,\,7$}
\rput[tl](2.75,-5.3){$Theorem\,\,4$}
\end{pspicture*}
}
\caption{Inclusion maps of $\mathbb{Np}$, $\mathbb{DS}$ and $\GB$}\label{fig:compare}
\end{center}
\end{figure}
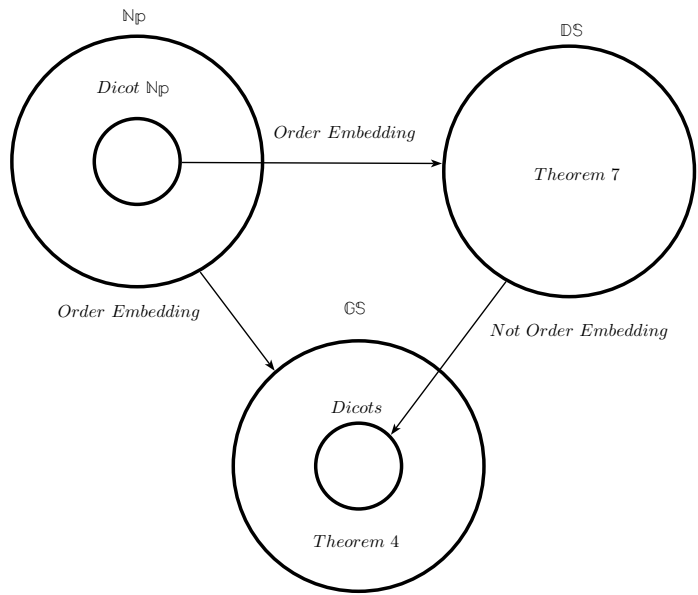

Although the inclusion map from Ettinger's universe to Guaranteed Scoring is not order preserving, some nice properties still hold, using a refined setting of pass-allowed stops; see Figure \ref{fig:compare}. See also item 4 in Section \ref{sec:Ste}.

\subsection{Stewart's general Scoring games}\label{sec:Ste}
The Scoring-universe that Stewart defined \cite{Stewa2011}, $\mathbb{S}' \subset \mathbb{S}$  (the only restriction is $\langle \, \emptyset^r\mid\emptyset^{s} \, \rangle\in \mathbb{S}'$ implies $r=s$), does allow non-stable games, in particular hot-atomic-games, such as
$\langle \, \emptyset^9\mid-5\, \rangle$.
It has some disadvantageous properties.\\

\noindent \textbf{1)}
 In $\mathbb{S}'$, we have $G=0\Rightarrow G\cong \langle \,   \emptyset^0\mid \emptyset^0  \, \rangle $, where `$\cong$' denotes identical game trees. (See the discussion before Corollary \ref{0corollary} about size of equivalence classes.)
 The argument is as follows. Suppose that $G=0$ is such that $\GL \neq\emptyset$ and consider the distinguishing
game $X=\langle \,   \emptyset^a\mid b  \, \rangle $, where $a>0$ and $b$ is less than all the real numbers
that occur in any follower of $G$. If Left starts in $G+X$ she loses; if Left starts in
$0+X$ she wins (\cite{Johns2014}, p.36). Therefore, $G\neq 0$.
This situation occurs because $X$ is a ``strange'' game
where Left wants to have the turn but she has no moves. So, the only invertible games
of $\mathbb{S}'$ are the numbers.\\

\noindent \textbf{2)} $\mathbb{S}'$ is non-natural.
In a natural universe we have $\Num{1}>0$. Recall,
$\zeta(1)=\langle \,   0 \mid \emptyset^0  \, \rangle =  \Num{1}$ and
$\zeta(0)=\Num{0}=0$.
Consider the distinguishing game $X=\langle \,   \emptyset^2\mid -3  \, \rangle $:
\begin{align*}
\LS(\langle \,   0 \mid \emptyset^0   \, \rangle + \langle \,   \emptyset^2 \mid -3  \, \rangle ) &= -3,\\
\LS(0+\langle \,   \emptyset^2 \mid -3   \, \rangle )&=2, \text{ since Left has no move.}
\end{align*}
Thus, by Definition \ref{def:equality}, $\Num{1}>0$ is not true in $\mathbb{S}'$.\\

\noindent \textbf{3)}
The Greediness principle fails in $\mathbb{S}'$.
Consider $G=\langle \,   0\mid \emptyset^0  \rangle$ and $H=\langle \,   \emptyset^0\mid \emptyset^0  \, \rangle $.
 There are instances when Left does not prefer $G$; for example, if $X=\langle \,   \emptyset^1\mid -1  \, \rangle $, then $\LS(H+X) = 1$, $\RS(H+X) = -1$ and $\LS(G+X) = -1$,
$\RS(G+X) = -1$. Thus, by Definition \ref{def:equality}, $G\not\geqslant  H$.
\\

\noindent  \textbf{4)} Clearly $\mathbb{DS}\subset \mathbb{S}$ but the inclusion
map is not order-preserving (see also Figure~\ref{fig:compare} for a diagram of the results in this paper). Consider the dicot game $G=\langle \,   \langle \,   1\mid 1  \, \rangle \mid \langle \, 1\mid 1  \, \rangle  \, \rangle$.
By Corollary \ref{Ettingercorollary}, $G>0$ in $\mathbb{DS}$.
However, in $\mathbb{S}$,
using the distinguishing game $X=\langle \,   \emptyset^{\frac{1}{2}}\mid \langle \,   \langle \,   -2\mid 2\, \rangle \mid -3  \, \rangle   \, \rangle $,
of course, $\LS(0+X) = \frac{1}{2} > 0$. But, in the game $G+X$, Left has only one legal move, that to $\langle \,   1\mid 1  \, \rangle + X$, and so
Right goes to $\langle \,   1\mid 1  \, \rangle + \langle \,   \langle \,   -2\mid 2  \, \rangle \mid -3  \, \rangle $, which gives $\LS(G'+X) = -1 < 0$. Thus, $G>0$ in $\mathbb{DS}$, but $G\ngeqslant0$ in $\mathbb{S}$.

\subsection{Johnson's well-tempered, dicot Scoring games}
Johnson's \cite{Johns2014} universe consists of dicot games in which, for a given game $G$,
the length of any play (distance to any leaf on the game tree) has the same parity.
The games are called \textit{even-tempered} if all the lengths are even and
called \textit{odd-tempered} otherwise.
A game, $G$ is \textit{inversive} if $\LS(G+X)\geqslant \RS(G+X)$ for every even-tempered
game $X$. Although the whole set of games is not well-behaved,
for example,  canonical forms do not exist, each inversive game has a canonical form and
an additive inverse which is equal to its conjugate; in fact they form an abelian group.
Moreover, $G\geqslant H$ if $G$ and $H$ have the same `temper' and $\RS(G-H)\geqslant 0$.

\section{Normal-play games}\label{sec:Normal-play}

The definitions for Normal-play are standard, along with other material,
in any of \cite{BerleCG2001-2004, AlberNW2007, Siege2013}.

Under Normal-play, there are four outcome classes:

  \small
  \begin{center}
  \begin{tabular}{|l|l|l|}
    \hline
    Class & Name & Definition \\
     \hline
    ${\cal N}$ & incomparable & The ${\cal N}$ext player wins\\
     \hline
    ${\cal P}$ & zero & The ${\cal P}$revious player wins\\
   &&(more precisely ${\cal N}$ext player loses) \\
     \hline
    ${\cal L}$ & positive & ${\cal L}$eft wins regardless of who plays first \\
     \hline
    ${\cal R}$ & negative & ${\cal R}$ight wins regardless of who plays first \\
    \hline
  \end{tabular}
  \end{center}
  \normalsize

  \noindent
  We write $\circ(G)$ to designate the outcome of $G$. The fundamental
  definitions of Normal-play structure are based in these outcomes.

  \begin{definition} \label{Equivalence} (Equivalence)
  $G=H\,\,\mathrm{if} \circ(G+X)=\circ(H+X)\,\,\mathrm{for}\,\mathrm{all}\,\mathrm{games}\,X.$
  \end{definition}

 The convention is that positive is good for Left and negative for Right and
  the outcomes are ordered: $\mathcal{L}$ is greater than
  both $\mathcal{N}$ and $\mathcal{P}$, which in turn are both greater than $\mathcal{R}$,
  and, finally, $\mathcal{N}$ and $\mathcal{P}$ are incomparable.
In Normal-play games, there is a way to check for the equivalence of games $G$ and $H$,
 which does not require considering any third game:
$$\textit{$G=H$ iff $G-H$ is a second player win. }$$

  \begin{definition} \label{Order} (Order) $G\geqslant  H$ if $\forall X$,  $o(G+X)\geqslant  o(H+X)$.
  \end{definition}

   \begin{definition} \label{Number} (Number) A game $G$ is a number if all
    $G^{\rm L}$ and $G^{\rm R}$ are numbers and
for each pair of options, $G^{\rm L}<G^{\rm R}$.
    \end{definition}

Note that $0=\{\emptyset \mid \emptyset\}$ is a number since there are no options to compare.
 If the number is positive, then this represents the number of moves advantage
 Left has over Right. An important concept is the `stop'---the
 best number that either player can achieve when going first.

  \begin{definition}\label{Stops} (Stops)
 The \emph{Left-stop} and the \emph{Right-stop} of a game $G$ are:
 \begin{eqnarray}
  LS(G) &=& \begin{cases}G             & \text{if $G$ is a number}, \\
                    \max(RS(G^{\rm L})) & \text{if $G$ is not a
  number};\end{cases} \\
  RS(G) &=& \begin{cases}G             & \text{if $G$ is a number}, \\
                    \min(LS(G^{\rm R})) & \text{if $G$ is not a number}.\end{cases}
  \end{eqnarray}
  \end{definition}

  There are two possible relations between the stops.
  One: $LS(G) > RS(G)$ and $G$ is referred to as a \textit{hot} game.
  The term was chosen to give the idea that the players really want to play first in $G$.
   Two: $LS(G) = RS(G)$ and $G$ is either a number or a \emph{tepid}
   game such as a position in \textsc{nim} and \textsc{clobber} \cite{AlberNW2007}. In the presence of a hot game,
    moves in a tepid game are not urgent, and moves in numbers are never urgent.
The situation $LS(G) < RS(G)$ does not occur, since then $G$ would be
 a Zugzwang game, but in Normal-play these games are numbers and so
 the relationship reverts to $LS(G) = RS(G)$.

The Lawyer's offer should obviously have been accepted if the question would have concerned Normal-play, because here the worst thing imaginable is to run out of options.

\bibliographystyle{plain}
\bibliography{games3}

\end{document}